\documentclass[letterpaper, 10 pt, conference]{ieeeconf}

\IEEEoverridecommandlockouts        
\usepackage{cite}
\usepackage{amsmath,amssymb,amsfonts}
\usepackage{hyperref}
\usepackage{epstopdf,color}
\usepackage{textcomp}
\usepackage{graphicx,color}
\graphicspath{{img/}}
\usepackage{mathrsfs}
\usepackage[vlined,ruled]{algorithm2e}
\usepackage{subfigure}
\usepackage{url}
\usepackage{color}
\usepackage{dsfont}
\usepackage{bbm}
\usepackage{booktabs}
\usepackage{array}
\usepackage[table]{xcolor}
\usepackage{yfonts}
\usepackage{cleveref} 
\usepackage{tikz}
\usepackage{pgfplots}
\usepackage{multirow}
\usepackage{textcomp}
\usepackage{tabularx}
\usepackage{placeins}
\usepackage{float}
\usepackage{nccmath}
\usepackage{accents}
\usepackage{multicol}
\usepackage{etoolbox, refcount}
\usepackage{chngcntr}
\usepackage{apptools}
\usepackage{relsize}
\usepackage[font=footnotesize]{caption}
\usepackage{MnSymbol}

\newtheorem{theorem}{Theorem}[section]
\newtheorem{lemma}[theorem]{Lemma}

\newtheorem{proposition}[theorem]{Proposition}

\newtheorem{remark}[theorem]{Remark}
\newtheorem{assumption}{Assumption}
\newtheorem{example}[theorem]{Example}

\newcommand{\setdef}[2]{\{#1 \; : \; #2\}}

\newcommand{\until}[1]{\{1,\dots,#1\}}

\newcommand{\Kc}{\mathcal{K}}

\newcommand{\Lc}{\mathcal{L}}

\newcommand{\Yc}{\mathcal{Y}}
\newcommand{\Zc}{\mathcal{Z}}

\newcommand{\real}{\mathbb{R}}

\newcommand{\Fc}{\mathcal{F}}
\newcommand{\Vc}{\mathcal{V}}
\newcommand{\Pc}{\mathcal{P}}
\newcommand{\Ic}{\mathcal{I}}
\newcommand{\Nc}{\mathcal{N}}
\newcommand{\Gc}{\mathcal{G}}
\newcommand{\Ec}{\mathcal{E}}

\newcommand{\argmin}[2] {\mathrm{arg}\min_{#1}#2}

\DeclareSymbolFont{bbold}{U}{bbold}{m}{n}
\DeclareSymbolFontAlphabet{\mathbbold}{bbold}

\newcommand{\norm}[1]{\left\lVert#1\right\rVert}

\newcommand\oprocendsymbol{\hbox{$\bullet$}}
\newcommand\oprocend{\relax\ifmmode\else\unskip\hfill\fi\oprocendsymbol}

\renewcommand{\theenumi}{(\roman{enumi})}

\newcommand*{\QEDA}{\hfill\ensuremath{\blacksquare}}%

\newcommand\xqed[1]{%
  \leavevmode\unskip\penalty9999 \hbox{}\nobreak\hfill
  \quad\hbox{#1}}
\newcommand\demo{\xqed{$\bullet$}}

\newcounter{countitems}
\newcounter{nextitemizecount}
\newcommand{\setupcountitems}{%
  \stepcounter{nextitemizecount}%
  \setcounter{countitems}{0}%
  \preto\item{\stepcounter{countitems}}%
}
\makeatletter
\newcommand{\computecountitems}{%
  \edef\@currentlabel{\number\c@countitems}%
  \label{countitems@\number\numexpr\value{nextitemizecount}-1\relax}%
}
\newcommand{\nextitemizecount}{%
  \getrefnumber{countitems@\number\c@nextitemizecount}%
}
\newcommand{\previtemizecount}{%
  \getrefnumber{countitems@\number\numexpr\value{nextitemizecount}-1\relax}%
}
\makeatother    
{\end{itemize}%
\unskip\computecountitems\ifnumcomp{\previtemizecount}{>}{4}{\end{multicols}}{}}

\newcommand{\longthmtitle}[1]{\mbox{}\emph{(#1):}}

\newcommand{\comment}[1]{} 
\newcolumntype{P}[1]{>{\centering\arraybackslash}p{#1}}
\allowdisplaybreaks
\renewcommand{\baselinestretch}{1}

\begin{document}
\title{\bf Distributed and Anytime Algorithm for Network Optimization
  Problems with Separable Structure\thanks{This work was partially
    supported by ARL award ARL-W911NF-22-2-0231.}}

\author{Pol Mestres \qquad Jorge Cort{\'e}s \thanks{P. Mestres and
    J. Cort\'es are with the Department of Mechanical and Aerospace
    Engineering, University of California, San Diego,
    \{pomestre,cortes\}@ucsd.edu}%
}

\maketitle

\begin{abstract}
  This paper considers the problem of designing a dynamical system to
  solve constrained optimization problems in a distributed way and in an
  anytime fashion (i.e., such that the feasible set is forward
  invariant). For problems with separable objective function and
  constraints, we design an algorithm with the desired properties and
  establish its convergence. Simulations illustrate our results.
\end{abstract}

\section{Introduction}\label{sec:introduction}

%
%
Distributed optimization methods are a popular tool for
solving several engineering problems like parameter estimation,
resource allocation in communication networks, source localization,
etc. In several of these applications, feedback controllers are often
implemented as the solution of such problems on a physical plant. In this context, the safe operation requirements of the system are encoded as constraints of the optimization problem. This approach is very versatile, but the implementation of such controllers faces several challenges.
%
%
On the one hand, the algorithm solving the optimization problem may be terminated at any time, and hence
feasibility must be maintained throughout its execution. We refer to an algorithm with this property as anytime. Moreover, it must retain its distributed and scalable character, so that each agent can implement it by communicating exclusively with its neighbors and can do so in an efficient manner independently of the size of the network. Designing algorithms that combine all these objectives together and are anytime, distributed, scalable and have provable convergence guarantees is a challenging problem.
%

\subsubsection*{Literature Review}
%
%
In this work we take the viewpoint of optimization algorithms as continuous-time dynamical systems (cf.~\cite{RWB:91,UH-JBM:94}), which has recently proven to be a very powerful paradigm in applications where the optimization problem is in a feedback loop with a plant~\cite{AH-SB-GH-FD:21}.
The problem of designing distributed algorithms for constrained
optimization is well studied in the
literature. The survey paper~\cite{AN-JL:18} covers this topic exhaustively. Of particular interest to us are primal-dual and projected saddle-point dynamics~\cite{KA-LH-HU:58}, which define dynamical systems that solve constrained optimization problems.
%
%
Although these works have provable convergence guarantees, the
trajectories generated by the proposed dynamical systems are in general not
guaranteed to be feasible throughout its execution.
This anytime property is well-studied in the optimization literature and dates back to the mid-1980s (cf.~\cite{TD:87}) in the context of time-dependent planning.
Our work here is related to~\cite{YSX-JW:00,MM-MIJ:22,AA-JC:24-tac}, which design dynamical systems that solve nonlinear programs in continuous time with the anytime property. However, the literature on anytime algorithms generally is not concerned in making them amenable to a distributed implementation. 
An exception is the relaxed economic dispatch problem, which involves a convex
separable objective function and globally coupling affine equality
constraints. For this problem,~\cite{AC-JC:15-tcns} gives a
distributed anytime algorithm that converges to the global optimum.
%
%
The problem of designing distributed anytime algorithms for
constrained optimization problems is very relevant in the control
barrier function (CBF) literature~\cite{ADA-SC-ME-GN-KS-PT:19,ADA-XX-JWG-PT:17}. 
This is because CBF-based controllers for multiagent systems can be obtained as the solution of a network optimization problem, where the system is guaranteed to be safe only if its constraints are satisfied at all times.
%
%
The works~\cite{UB-LW-ADA-ME:15,LW-ADA-ME:17,MJ-MS:21} tackle this problem for CBF-based quadratic programs (QPs), where a centralized QP is split into local QPs whose solution is guaranteed to preserve the safety constraints at all times.
%
%
However, the solution of these local QPs might be suboptimal with
respect to the centralized QP. The recent work~\cite{XT-DVD:22}
designs a distributed algorithm that is guaranteed to satisfy the
constraints of the CBF-based QP at all times
and converge to its state-dependent optimizer in finite time. However,
their algorithm is restricted to a limited
%
%
class of quadratic programs and plant dynamics and is not easily
generalizable to general convex programs.
%
%

\subsubsection*{Statement of Contributions}
%
%

In this paper we introduce a continuous-time dynamical system to solve
convex optimization problems with separable objective function and
constraints in a distributed and anytime fashion.  The constraints
couple the decision variables of all agents and this poses a
difficulty in the design of distributed algorithms that solve such
problems.  We first show that the separable structure permits the
intoduction of auxiliary variables to reformulate the original problem
into one with local constraints while still preserving the same
solution set. However, this reformulation still does not allow to
fully decouple the optimization problem into one per agent because the
auxiliary variables require coordination.  In order to sort this
hurdle, our technical approach constructs a dynamical system by
combining the use of projected saddle-point dynamics,
%
%
which are distributed but not anytime, and the safe gradient flow,
which is anytime but not distributed.
%
%
First, we establish the well-posedness
of the proposed dynamical system.
%
%
Second, we show that it is distributed, exhibits the anytime property
and is scalable.
%
%
Finally, we prove that all trajectories with feasible initial
condition converge to a neighborhood of the optimizer, which can be
made arbitrarily small by tuning a design parameter
accordingly. Moreover, in the case where the feasible set is bounded,
we show that all trajectories with feasible initial condition exactly
converge to the optimizer.
%
%
For reasons of space, proofs are omitted and will appear elsewhere.

\section{Preliminaries}\label{sec:prelims}

This section presents background on dynamical
systems that solve constrained optimization problems\footnote{ Throughout the paper we denote by $\real$ and $\real_{>0}$ the set of real and nonnegative
real numbers, respectively. Given $x\in\real^{n}$, $\norm{x}$ denotes
its Euclidean norm. For $a\in\real$ and $b\in\real_{\geq0}$, we let
\begin{align*}
  [a]_b^+ = \begin{cases}
    a, \quad &\text{if} \ b>0, \\
    \max\{0,a\}, \quad &\text{if} \ b=0.
\end{cases}
\end{align*}
For vectors $a\in\real^{n}$ and $b\in\real^{n}_{\geq0}$, $[a]_{b}^{+}$
denotes the vector whose $i$-th component is $[a_i]_{b_i}^{+}$, for
$i\in\{1,\dots,n\}$. We also write
$\textbf{0}_{n}=(0,\dots,0)\in\real^{n}$.  For a real-valued function
$F:\real^{n}\times\real^{m}\to\real$, we denote by $\nabla_{x}F$ and
$\nabla_{y}F$ the column vector of partial derivatives of $F$ with
respect to the first and second arguments, respectively.  Given a set
of functions $g^{1},\dots,g^{k}$, we let
$I_{g^1,\dots,g^k}(x)=\setdef{1\leq i\leq k}{g^i(x)=0}$.  For
matrices $A\in\real^{n\times m}$ and $B\in\real^{p\times q}$, we let
$A \otimes B$ denote their Kronecker product. Given a set
$\Pc\subseteq\real^{n}$ and a set of variables
$\xi=\{ x_{i_1},x_{i_2},\dots,x_{i_k} \}$, we denote by
$\Pi_{\xi}\Pc=\setdef{(x_{i_1},x_{i_2},\dots,x_{i_k})\in\real^k}{x\in\Pc}$
the projection of $\Pc$ onto the $\xi$ variables.  Given sets
$S_{1}, \dots, S_k$, $\bigtimes_{i=1}^{k}S_{i}$ denotes their
Cartesian product. An undirected graph is a pair $\Gc=(V,\Ec)$, where $V=\until{N}$ is a
finite set called the vertex set, $\Ec\subseteq V\times V$ is called
the edge set where $(i,j)\in\Ec$ if and only if $(j,i)\in\Ec$. The set
of neighbors of node $i$ is denoted by
$\Nc_{i}=\setdef{j\in\Ic}{(i,j)\in\Ec}$. The adjacency matrix
$A\in\real^{|V|\times|V|}_{\geq0}$ of the graph $\Gc$ satisfies the
property $[A]_{i,j}=[A]_{j,i}=1$ if $(i,j)\in\Ec$ and $[A]_{i,j}=0$
otherwise. The degree matrix $D$ of $\Gc$ is the diagonal matrix
defined by $[D]_{i,i}=|\Nc_i|$ for all $i\in\{1,\dots,|V|\}$. The
Laplacian matrix of $\Gc$ is $L=D-A$. 
Let $F:\real^{n}\to\real^{n}$ be a locally Lipschitz vector field and
consider the dynamical system $\dot{x}=F(x)$, with flow map
$\Phi_{t}:\real^{n}\to\real^{n}$. This means that $\Phi_{t}(x)=x(t)$,
where $x(t)$ is the unique solution of the dynamical system with
$x(0)=x$. A set $\Kc\subseteq\real^{n}$ is (positively) forward
invariant if $x\in\Kc$ implies that $\Phi_{t}(x)\in\Kc$ for all
$t\geq0$. A set $A\subseteq\Kc$ is Lyapunov stable relative to $\Kc$
if, for every open set $U$ containing $A$, there exists an open set
$\tilde{U}$ also containing $A$ such that for all
$x\in\tilde{U}\cap\Kc$, $\Phi_{t}(x)\in U\cap\Kc$ for all $t\geq0$.  A
set $A\subseteq\Kc$ is asymptotically stable relative to $\Kc$ if it
is Lyapunov stable relative to $\Kc$ and there is an open set $U$
containing $A$ such that $\Phi_{t}(x)\to A$ as $t\to\infty$ for
all~$x\in U\cap\Kc$.}.

\subsubsection*{Safe Gradient Flow}
Given continuously differentiable functions $f:\real^{n}\to\real$,
$g:\real^{n}\to\real^{p}$, $h:\real^{n}\to\real^{q}$, consider the
constrained nonlinear program
\begin{align}\label{eq:constrained-opt}
  & \min \limits_{x\in\real^n} f(x),
  \\
  \notag
  \qquad & \text{s.t.} \quad g(x) \leq0, \\
  \notag
  \qquad & \qquad h(x) = 0.
\end{align}
Let $\Fc=\setdef{x\in\real^n}{g(x)\leq0, \ h(x)=0}$ denote the
constraint set and
$ X_{\text{KKT}}=\setdef{x^*\in\real^n}{\exists
  (u^*,v^*)\in\real^p\times\real^q \ \text{such that} \ (x^*,u^*,v^*) \
  \text{is a KKT point of}~\eqref{eq:constrained-opt}}$.  We are
interested in solving the optimization problem with an algorithm that
respects its constraints at all times of its evolution. Given
$\alpha>0$, the \textit{safe gradient flow}, cf.~\cite{AA-JC:24-tac},
is the dynamical system $ \dot{x}=F_{\alpha}(x)$, where
\begin{align}\label{eq:sgf-opt-pb}
  \notag
  F_{\alpha}(x)
  &= \argmin{\xi\in\real^n}{\frac{1}{2}\norm{\xi+\nabla
    f(x)}^2}
  \\
  & \text{s.t.} \quad \frac{\partial g(x)}{\partial x}\xi \leq -\alpha
    g(x),
  \\
  \notag
  & \qquad \frac{\partial h(x)}{\partial x}\xi = -\alpha h(x).
\end{align}
The direction prescribed by $ F_{\alpha}$ can be interpreted as that
closest to the gradient descent direction $-\nabla f$ while ensuring
that the constraints defining $\Fc$ are not violated.  The next result
gathers important properties of the safe gradient flow.

\begin{proposition}\longthmtitle{Properties of the safe gradient
    flow~\cite[Proposition 5.1, Proposition 5.6 and Corollary
    5.9]{AA-JC:24-tac}}\label{properties-sgf}
  Suppose $f$, $g$ and $h$ are continuously differentiable and their
  derivatives are locally Lipschitz.  Then,
  \begin{enumerate}
  \item there exists an open neighborhood $U$ containing $\Fc$ such
    that~\eqref{eq:sgf-opt-pb} is well-defined;
  \item $F_{\alpha}$ is locally Lipschitz on $U$;
  \item the Lagrange multipliers of~\eqref{eq:sgf-opt-pb} are unique
    and locally Lipschitz as a function of $x$ on $U$;
  \item the feasible set $\Fc$ is forward invariant and asymptotically stable
    under~\eqref{eq:sgf-opt-pb};
  \item $F_{\alpha}(x^*)=0$ if and only if $x^{*}\in X_{\text{KKT}}$;
  \item if $x^{*}$ is a strict local minimizer of $f$ and an isolated
    equilibrium of the safe gradient flow, then $x^{*}$ is
    asymptotically stable relative to $\Fc$.
  \end{enumerate}
\end{proposition}

\subsubsection*{Projected Saddle-Point Dynamics}
We recall here the notion of projected saddle-point dynamics
following~\cite{AC-EM-SHL-JC:18-tac}. Consider again the optimization
problem~\eqref{eq:constrained-opt}, with continuously differentiable
functions $f$, $g$ and $h$ whose derivatives are locally Lipschitz,
and let $\Lc:\real^{n}\times\real^{m}_{\geq0}\times\real^{p}$ be the
associated Lagrangian,
\begin{align*}
  \Lc(x,\lambda,\mu)=f(x)+\lambda^{T}g(x)+\mu^{T} h(x).
\end{align*}
We define the \textit{projected saddle-point dynamics} for $\Lc$ as:
\begin{subequations}
  \begin{align}
    \dot{x} &= -\nabla_x \Lc(x,\lambda,\mu),
    \\
    \dot{\lambda} &= [\nabla_{\lambda}
                    \Lc(x,\lambda,\mu)]_{\lambda}^{+},
    \\
    \dot{\mu} &= \nabla_{\mu} \Lc(x,\lambda,\mu).
  \end{align}
  \label{eq:projected-saddle-point-dyn}
\end{subequations}
If $f$ is strongly convex, $g$ is convex and $h$ is affine, the saddle
point of $\Lc$ is unique and corresponds to the KKT point
of~\eqref{eq:constrained-opt}.  Moreover,~\cite[Theorem
5.1]{AC-EM-SHL-JC:18-tac} ensures that the saddle point of $\Lc$ is
globally asymptotically stable under the
dynamics~\eqref{eq:projected-saddle-point-dyn}.

\section{Problem Statement}\label{sec:problem-statement}
Consider a network composed by agents $\until{N}$ whose communication
topology is described by a connected undirected graph $\Gc$. An edge
$(i,j)$ represents the fact that agent $i$ can receive information
from agent $j$ and vice versa. We refer to an algorithm run by the
network as \emph{distributed} if each agent can execute it with the
information available to it and its neighbors.

For each $i\in \{1,\dots,N\}$, $k\in\{1,\dots,p\}$,
$l\in\{1,\dots,q\}$ let $f_{i}:\real^{n}\to\real$ be a strongly convex
and continuously differentiable function with locally Lipschitz
derivatives, $g_{i}^k:\real^{n}\to\real$ be a convex and continuously
differentiable function with locally Lipschitz derivatives and
$h_{i}^{l}:\real^{n}\to\real$ be an affine function. We let
$x=[x_1,\dots,x_N]\in\real^{nN}$. Consider the following optimization
problem with separable objective function and constraints:
\begin{align}\label{eq:distr-opt-pb}
  & \min \limits_{ x \in\real^{nN} } \sum_{i=1}^N f_i(x_i),
  \\
  \notag
  \qquad & \text{s.t.} \quad \sum_{i=1}^N g_i^{k}(x_i) \leq0, \quad
           k\in\{1,\dots,p\},
  \\
  \notag
  \qquad & \qquad \sum_{i=1}^N h_i^l(x_i) = 0, \quad l\in\{1,\dots,q\}.
\end{align}
Since the objective function is strongly convex and the feasible set
is convex, this program has a unique optimizer~$x^{*}$.  Note that,
even though the objective function is separable, the structure of the
constraints couples the decision variables of the agents. This makes
challenging the design of distributed algorithmic solutions
of~\eqref{eq:distr-opt-pb}.

\begin{remark}\longthmtitle{Separability structure}
  Problems of the form~\eqref{eq:distr-opt-pb} arise in multiple
  applications, including communications~\cite{FPK-AKM-DKHT:98},
  economic dispatch of power systems~\cite{AC-JC:15-tcns}, optimal
  power flow~\cite{TE:14},
  resource allocation~\cite{LX-SB:06},
  and safe swarm behavior using control barrier
  functions~\cite{UB-LW-ADA-ME:15}. 
  Also, given convex sets $X_i$, $i\in\{1,\dots,N\}$, a common
  problem considered in the distributed optimization
  literature~\cite{AN-AO-PAP:10} is
  \begin{align*}
    & \min \limits_{ x\in\real^n } \sum_{i=1}^N f_i(x),
    \\
    \notag
    \qquad & \text{s.t.} \quad x\in\cap_{i=1}^N X_i.
  \end{align*}
  When
  $X_{i}=\setdef{x\in\real^n}{\bar{g}_i(x)\leq0}\subseteq\real^{n}$
  for a continuously differentiable convex function with locally Lipschitz derivatives
  $\bar{g}_{i}:\real^{n}\to\real^{m_i}$, for $i\in\{1,\dots,N\}$, the
  optimization can be reformulated as
  \begin{align*}
    & \min \limits_{ x\in\real^{nN} } \sum_{i=1}^N f_i(x_i),
    \\
    \notag
    \qquad & \text{s.t.} \quad (L \otimes I_n ) x = \textbf{0}_{Nn} \quad \bar{g}_i(x_i)\leq0, \
             i\in\{1,\dots,N\},
  \end{align*}
  which is of the form~\eqref{eq:distr-opt-pb}.  \demo
\end{remark}

Throughout the paper, we denote $f(x) = \sum_{i=1}^N f_i(x_i)$,
$g^k(x) = \sum_{i=1}^N g_i^{k}(x_i)$ for $k\in\{1,\dots,p\}$ and
$h^{l}(x)=\sum_{i=1}^{N} h_{i}^{l}(x_i)$ for $l\in\{1,\dots,q\}$, and
write the feasible set of~\eqref{eq:distr-opt-pb} as
\begin{align*}
  \Fc = \setdef{ x \in\real^{nN}}{
  & g^k(x)\leq0, \ \forall
    k\in\{1,\dots,p\},
  \\
  & h^l(x) = 0, \ \forall l\in\{1,\dots,q\} }.
\end{align*}

We also make the following assumption.

\begin{assumption}\longthmtitle{Linear independence constraint
    qualification for separable constraints}\label{as:licq}
  For all $x\in\real^{nN}$, the vectors
  $\{\nabla g^k(x)\}_{k\in I_{g^1,\dots,g^p}(x)} \cup \{\nabla
  h^l(x)\}_{1\leq l \leq q}$ are linearly independent.
\end{assumption}

Assumption~\ref{as:licq} is common and guarantees that the KKT
conditions are necessary and sufficient for the optimality
of~\eqref{eq:distr-opt-pb}.

Our goal is to design an algorithm, in the form of a locally Lipschitz
dynamical system, such that
\begin{enumerate}
\item is \emph{distributed}, i.e., each agent can execute it with
  locally available information;
\item is \emph{anytime}, i.e., the feasible set $\Fc$ is forward
  invariant;
\item \emph{solves}~\eqref{eq:distr-opt-pb}, i.e., all trajectories
  starting in $\Fc$ converge to its optimizer.
\end{enumerate}
Even though algorithmic solutions exist in the literature that enjoy
some of these properties (e.g., the projected saddle-point dynamics
enjoys (i) and (iii) for certain classes of optimization problems),
the design of an algorithm that enjoys all three is challenging.

\section{Design of Algorithmic Solution}\label{sec:separable}
Here we propose an algorithmic solution to the constrained
program~\eqref{eq:distr-opt-pb} to meet the requirements stated in
Section~\ref{sec:problem-statement}.  Our exposition proceeds by first
reformulating the optimization problem and then building on the
projected saddle-point dynamics and the safe gradient flow to
synthesize a coordination algorithm with the desired properties.

\subsection{Reformulation using constraint mismatch variables}

In this section we provide an equivalent formulation
of~\eqref{eq:distr-opt-pb} that addresses the coupling among the
agents' decision variables arising from the structure of the
constraints.  The basic idea to ``decouple'' them is to introduce,
following~\cite{AC-JC:16-allerton}, \emph{constraint-mismatch
  variables} which help agents keep track of local constraints while
collectively satisfying the original constraints. Formally, to the
state of each agent, we add one variable per constraint: $y_{i}^k$ for
agent $i$ and the $k$th inequality constraint and $z_{j}^l$ for agent
$j$ and the $l$th equality constraint.  For convenience, we use the
notation $x=[x_1,\dots,x_N]$, $y_i=[y_i^1,\dots,y_i^p]$,
$z_{i}=[z_i^1,\dots,z_i^p]$, $y=[y_1,\dots,y_N]$, $z=[z_1,\dots,z_N]$.
Consider then the following problem
\begin{align}\label{eq:distr-opt-pb-y}
  & \min \limits_{ x \in\real^{nN}, y \in\real^{Np}, z\in\real^{Nq} } \sum_{i=1}^N f_i(x_i),
  \\
  \notag
  & \qquad \text{s.t.} \quad g_i^k(x_i)+\sum_{j\in
    \Nc_i}(y_i^k-y_j^k)\leq0, \quad k\in\{1,\dots,p\}
  \\ 
  \notag
  & \qquad \qquad \quad h_i^l(x_i)+\sum_{j\in\Nc_i}(z_i^k-z_j^k)=0, \quad l\in\{1,\dots,q\},
\end{align}
and constraints for all $i\in\{1,\dots,N\}$.
Note that, in this formulation, constraints are now locally
expressible, meaning that agent $i \in \until{N}$ can evaluate the
ones corresponding to $g_i^k$ and $h_i^l$ with information from its
neighbors.  Let $\mu_{i}=[\mu_i^1,\dots,\mu_i^p]$,
$\lambda_{i}=[\lambda_i^1,\dots,\lambda_i^p]$,
$\lambda=[\lambda_1,\dots,\lambda_N]$ and $\mu=[\mu_1,\dots,\mu_N]$
be the Lagrange multipliers for the constraints
in~\eqref{eq:distr-opt-pb-y}. 
%
%
We next show that
the optimizer in $x$ of~\eqref{eq:distr-opt-pb-y} is~$x^{*}$, the
optimizer of~\eqref{eq:distr-opt-pb}.

\begin{proposition}\longthmtitle{Equivalence between the two
    formulations}\label{prop:equivalence-constraint}
  Let $\Fc_{r}^{*}$ be the solution set
  of~\eqref{eq:distr-opt-pb-y}. Then, $x^{*}=\Pi_{x}(\Fc_r^{*})$.
\end{proposition}
\begin{proof}
  Note that~\eqref{eq:distr-opt-pb} is equivalent to
  \begin{align}\label{eq:slack}
    \notag
    & \min \limits_{ x\in\real^{nN}, s\in\real^{p} } \sum_{i=1}^N f_i(x_i),
    \\
    \notag
    \qquad & \text{s.t.} \quad \sum_{i=1}^N g_i^k(x_i) + s^k = 0, \
             s^k \geq0,
    \\
    \qquad & \quad \quad \sum_{i=1}^N h_i^l(x_i) = 0,
    \\
    \notag
    \qquad & \quad \quad k\in\{1,\dots,p\}, \ l\in\{1,\dots,q\}.
  \end{align}
  and~\eqref{eq:distr-opt-pb-y} is equivalent to
  \begin{align}\label{eq:cm-slack}
    \notag
    & \min \limits_{ x\in\real^{Nn}, y\in\real^{Np}, s\in\real^{Np},
      z\in\real^{Nq} } \sum_{i=1}^N f_i(x_i), 
    \\
    \notag
    & \qquad \text{s.t.} \quad g_i^k(x_i)+\sum_{j\in
      \Nc_i}(y_i^k-y_j^k) +s_i^k=0, \ s_i^k\geq0, \\ 
    & \qquad \quad \quad h_i^l(x_i) + \sum_{j\in\Nc_i}(z_i^l-z_j^l) = 0, \\
    \notag
    &\qquad i \in \{1,\dots,N\}, \ k\in\{1,\dots,p\}, \ l\in\{1,\dots,q\}.
  \end{align}
  Now the proof follows a similar reasoning as the proof
  from~\cite[Proposition 4.2]{AC-JC:16-allerton}. We only need to show
  that the feasible sets of~\eqref{eq:slack} and~\eqref{eq:cm-slack}
  are the same, because their objective functions coincide.  First, if
  $(\hat{x},\hat{y},\hat{s},\hat{z})\in\real^{Nn}\times\real^{Np}\times\real^{Np}\times\real^{Nq}$
  is a feasible point for~\eqref{eq:cm-slack}, then by adding up all
  constraints for $i\in\until{N}$ and letting $\bar{s}^{k}=\sum_{i=1}^{N}\hat{s}_{i}^{k}$, $\bar{s}=[\bar{s}^1,\dots,\bar{s}^p]$ it follows that $(\hat{x},\bar{s})$
  is a feasible point for~\eqref{eq:slack}. Now, let
  $(\tilde{x},\tilde{s})$ be a feasible point
  for~\eqref{eq:slack}. Let
  $v = [g_1^1(\tilde{x}_1), \dots,
  g_i^k(\tilde{x}_N), \dots;
  g_N^p(\tilde{x}_N)]\in\real^{Np}$ and $\breve{s}=[\frac{\tilde{s}^1}{N},\dots,\frac{\tilde{s}^1}{N},\dots,\frac{\tilde{s}^p}{N},\dots,\frac{\tilde{s}^p}{N}]\in\real^{Np}$.
  Note that
  $\textbf{1}_{Np}^{T}(v+\breve{s})=0$. This implies that $v+\breve{s}$ belongs to the range
  space of the Laplacian $L$ of the communication graph and hence
  there exists $\tilde{y}$ such that $-L\tilde{y}=v+\breve{s}$. By a similar
  argument, by letting
  $w = [h_1^1(\tilde{x}_1), \dots, h_i^l(\tilde{x}_N), \dots;
  g_N^p(\tilde{x}_N)]$ there exists $\tilde{z}$ such that
  $-L\tilde{z}=w$. Now it follows that
  $(\tilde{x},\tilde{y},\breve{s},\tilde{z})$ is feasible
  for~\eqref{eq:slack}, hence proving that the feasible sets
  of~\eqref{eq:slack} and~\eqref{eq:cm-slack} are the same.
\end{proof}
\smallskip


Proposition~\ref{prop:equivalence-constraint} implies
that~\eqref{eq:distr-opt-pb-y} has a unique optimizer in the variables
$x$. However, since the objective function
in~\eqref{eq:distr-opt-pb-y} is not strongly convex in $y$ and $z$,
the optimizer in the variables $y$ and $z$ might not be unique. Hence,
for the results that follow, we take $\epsilon>0$ and define
$f_{i}^{\epsilon}(x_i,y_i,z_i) =
f_{i}(x_i)+\frac{\epsilon}{2}\norm{y_{i}}^{2}+\frac{\epsilon}{2}\norm{z_i}^{2}$,
$f^{\epsilon}(x,y,z)=\sum_{i=1}^{N}f_{i}^{\epsilon}(x_i,y_i,z_i)$.
Consider the following regularized version
of~\eqref{eq:distr-opt-pb-y},
\begin{align}\label{eq:distr-opt-pb-y-eps}
  & \min \limits_{ x\in\real^{Nn}, y\in\real^{Np}, z\in\real^{Nq} }
    \sum_{i=1}^N f_{i}^{\epsilon}(x_i, y_i, z_i), 
  \\
  \notag
  & \text{s.t.} \quad g_i^k(x_i)+\sum_{j\in \Nc_i}(y_i^k-y_j^k) \leq0, \quad k\in\{1,\dots,p\}, \\
  \notag
  & \qquad h_i^l(x_i)+\sum_{j\in\Nc_i}(z_i^k-z_j^k) = 0, \quad l\in\{1,\dots,q\},
\end{align}
and with constraints for all $i\in\{1,\dots,N\}$. Let
$(x^{*,\epsilon},y^{*,\epsilon},z^{*,\epsilon})\in\real^{Nn}\times\real^{Np}\times\real^{Nq}$
be the optimizer of~\eqref{eq:distr-opt-pb-y-eps}, which is unique
because the objective function is strongly convex and the feasible set
is convex. Next we establish a sensitivity result for the regularized
problem~\eqref{eq:distr-opt-pb-y-eps}.

\begin{lemma}\longthmtitle{Sensitivity of regularized
      problem}\label{lem:sensitivity-regularized}
    Given $\delta>0$, there exists $\bar{\epsilon}>0$ so that if
    $\epsilon<\bar{\epsilon}$, then
    $\norm{x^{*,\epsilon}-x^*}<\delta$.
\end{lemma}
\begin{proof}
  Let $(x^*,y^*,z^{*})$ be an optimizer of~\eqref{eq:distr-opt-pb-y}
  with $m = f(x^*,y^*,z^*)$.  Since $x^{*}$ is unique, there exists
  $\beta>0$ such that $ f(x)=f^0(x,y,z) \geq m+\beta$, for all
  $x\in\Fc$, $y\in\real^{Np}$, $z\in\real^{Nq}$ whenever
  $\norm{x-x^{*}}=\delta$.  Hence, for any $\epsilon>0$ it follows
  that
  \begin{align*}
    f^{\epsilon}(x,y,z) \geq f^{0}(x,y,z) \geq m + \beta,
  \end{align*}
  for all $x\in\Fc$, $y\in\real^{Np}$, $z\in\real^{Nq}$ whenever
  $\norm{x-x^{*}}=\delta$.  On the other hand, since $f$ is continuous
  with respect to $\epsilon$, for any $\delta>0$ we can find
  $\bar{\epsilon}$ such that
  \begin{align*}
    f^{\epsilon}(x^*,y^*,z^*) &\leq m+\frac{\beta}{2} \quad \forall \epsilon<\bar{\epsilon}.
  \end{align*}
  Hence, by taking $\epsilon<\bar{\epsilon}$ we can ensure that the set
  \begin{align*}
    \setdef{(x,y,z)\in\Fc\times\real^{Np}\times\real^{Nq}}{\norm{x-x^{*}}\leq\delta}
  \end{align*}
  contains the local minimizer of $f^{\epsilon}$ for any
  $\epsilon<\bar{\epsilon}$. Thus,
  $\norm{x^{*,\epsilon}-x^{*}}\leq\delta$ for all
  $\epsilon<\bar{\epsilon}$, as stated.
\end{proof}

\smallskip

Given Lemma~\ref{lem:sensitivity-regularized}, in what follows we
focus on solving~\eqref{eq:distr-opt-pb-y-eps} and assume that
$\epsilon$ is taken sufficiently small to guarantee a desired maximum
distance between $x^{*,\epsilon}$ and~$x^{*}$.

\subsection{Cascade of saddle-point dynamics and safe gradient flow}

In this section, we build on the reformulation presented above to
design our proposed algorithmic solution. Note that, if we had
knowledge of the optimizers $y^{*,\epsilon}, z^{*,\epsilon}$ of
Problem~\eqref{eq:distr-opt-pb-y-eps}, we could break the optimization
into $N$, one per agent $i\in\until{N}$, decoupled optimization
problems as follows,
\begin{align}\label{eq:local-distr-opt-pb-y-eps}
  & \min \limits_{ x_i\in\real^{n} } f_{i}^{\epsilon}(x_i),
  \\
  \notag
  & \text{s.t.} \quad g_i^k(x_i)+\sum_{j\in
    \Nc_i}((y_i^k)^{*,\epsilon}-(y_j^k)^{*,\epsilon}) \leq0, \quad k\in\{1,\dots,p\},
  \\
  \notag
  & \quad \quad h_i^l(x_i)+\sum_{j\in\Nc_i}((z_i^k)^{*,\epsilon}-(z_j^k)^{*,\epsilon})
    = 0, \quad l\in\{1,\dots,q\}.
\end{align}
In turn, each of these problems could be solved in an anytime fashion
by having each agent execute the corresponding safe gradient flow,
cf. Proposition~\ref{properties-sgf}.  However, since $y^{*,\epsilon}$
and $z^{*,\epsilon}$ are not readily available, agents need to
interact with their neighbors to compute them.  Since this would
require an iterative algorithm, this means agents will face evolving
$y$ and $z$ in the corresponding formulation
of~\eqref{eq:local-distr-opt-pb-y-eps}, which raises the additional
challenge of ensuring the anytime nature of the safe gradient flow is
preserved. We tackle these challenges next.

To generate the update law for $y$ and $z$, we propose to use the
projected saddle-point dynamics of~\eqref{eq:distr-opt-pb-y-eps}.
By~\cite[Theorem 5.1]{AC-EM-SHL-JC:18-tac}, these are guaranteed to
converge to its optimizers. Simultaneously, we implement the safe
gradient flow of~\eqref{eq:local-distr-opt-pb-y-eps} with the current
values of $y$ and $z$, i.e., (with the notation
$y_{\Nc_i}=y_i\cup\{y_j\}_{j\in\Nc_i}$,
$z_{\Nc_i}=z_i\cup\{z_j\}_{j\in\Nc_i}$):
\begin{align}\label{eq:sgf-i}
  \notag
  &S_{\alpha}^i(x_i,y_{\Nc_i},z_{\Nc_i}) =
  \\
  &\argmin{\xi_i\in\real^{n}}{\frac{1}{2}\norm{\xi_i+\nabla f_i(x_i)}^2},
  \\
  \notag
  & \text{s.t.} \ \nabla g_i^k(x_i)\xi_i \leq
    -\alpha(g_i^k(x_i)+\sum_{j\in \Nc_i}(y_i^k-y_j^k)), \ k\in\{1,\dots,p\},
  \\
  \notag
  & \quad \ \nabla h_i^l(x_i)\xi_i = -\alpha(h_i^l(x_i)+\sum_{j\in
    \Nc_i}(z_i^l-z_j^l)), \ l\in\{1,\dots,q\},
\end{align}
for all $i\in\{1,\dots,N\}$. We denote
$S_{\alpha}(x,y,z) = [S_{\alpha}^1(x_1,y_{\Nc_1},z_{\Nc_1}),
\dots,S_{\alpha}^N(x_N,y_{\Nc_N},z_{\Nc_N})]$.  To add more
flexibility to our design, we add a timescale separation parameter
$\tau>0$ that allows the projected saddle-point dynamics to be run at
a faster rate relative to the safe gradient flow.  This leads to the
cascaded dynamical system:
\begin{subequations}
  \begin{align}
    \tau \dot{v}_i &= -\nabla f_i(v_i)-\sum_{k=1}^p \lambda_i^k \nabla
                     g_i^k(v_i)- \sum_{l=1}^q \mu_i^l h_i^l(v_i),\label{eq:p-d-sgf-v}
    \\
    \tau \dot{y}_i^k &= -\epsilon y_i^k -\sum_{j\in \Nc_i}
                       (\lambda_i^k-\lambda_j^k),\label{eq:p-d-sgf-y}
    \\
    \tau \dot{z}_i^l &= -\epsilon z_i^l -\sum_{j\in\Nc_i}
                       (\mu_i^l-\mu_j^l),\label{eq:p-d-sgf-z}
    \\ 
    \tau \dot{\lambda}_i^k &= [g_i^k(v_i)+\sum_{j\in
                             \Nc_i}(y_i^k-y_j^k)]_{\lambda_i^k}^+,\label{eq:p-d-sgf-lambda}
    \\
    \tau \dot{\mu}_i^l &=
                         h_i^l(v_i)+\sum_{j\in\Nc_i}(z_i^k-z_j^k),\label{eq:p-d-sgf-mu}
    \\
    \dot{x}_i &= S_{\alpha}^i(x_i,y_{\Nc_i},z_{\Nc_i}),\label{eq:p-d-sgf-x}
  \end{align}
  \label{eq:primal-dual-sgf}
\end{subequations}
for $i\in\{1,\dots,N\}$, $k\in\{1,\dots,p\}$ and $l\in\{1,\dots,q\}$,
where $v_{i}$ are \textit{virtual} variables that play the role of
$x_{i}$ in the projected saddle-point dynamics.
Since~\eqref{eq:primal-dual-sgf} results from the cascaded
interconnection of saddle-point dynamics and the safe gradient flow,
we refer to it as SP-SGF.

\begin{remark}\longthmtitle{Scalability and distributed character of SP-SGF}\label{rem:scalable-distr}
  In algorithm~\eqref{eq:primal-dual-sgf}, each agent has a state
  variable of dimension~$2n+2p+2q$.  To compute the evolution of these
  state variables, each agent only requires information provided by
  its neighbors in~$\Gc$. Therefore, the algorithm is distributed. In
  addition, since the memory needed by each agent to
  run~\eqref{eq:primal-dual-sgf} remains constant as the network size
  $N$ increases, the algorithm is also scalable.  \demo
\end{remark}


\begin{remark}\longthmtitle{Algorithm implementation}
  Note that the execution of SP-SGF requires
  solving the optimization problem~\eqref{eq:sgf-i}, for each
  $i\in\{1,\dots,N\}$,
  which is a quadratic program and hence can be solved efficiently. In
  fact, if the number of constraints is low, closed-form expressions
  for its solution~\cite[Theorem 1]{XT-DVD:21} are available.  \demo
\end{remark}

In what follows, we assume that for all $i\in\{1,\dots,N\}$, the set
of initial conditions for $v_{i}$, $y_{i}$, $z_{i}$, $\lambda_{i}$ and
$\mu_{i}$ in~\eqref{eq:primal-dual-sgf} lie in compact sets $\Vc_{i}$,
$\Yc_{i}$, $\Zc_{i}$, $\Lambda_{i}$ and $M_{i}$ respectively. This
means that the initial conditions $v, y, z, \lambda, \mu$ lie in
compact sets $\Vc:=\bigtimes_{i=1}^{N}\Vc_{i}$,
$\Yc:=\bigtimes_{i=1}^{N}\Yc_{i}$, $\Zc:=\bigtimes_{i=1}^{N}\Zc_{i}$,
$\Lambda=\bigtimes_{i=1}^{N}\Lambda_{i}$,
$M=\bigtimes_{i=1}^{N}M_{i}$. Since the projected saddle-point
dynamics~\eqref{eq:p-d-sgf-v}-\eqref{eq:p-d-sgf-mu} are convergent
by~\cite[Theorem 5.1]{AC-EM-SHL-JC:18-tac}, there exist compact sets
$\bar{\Vc}$, $\bar{\Yc}$, $\bar{\Zc}$, $\bar{\Lambda}$, $\bar{M}$ such
that the trajectories of $v, y, z, \lambda, \mu$
under~\eqref{eq:primal-dual-sgf} stay in $\bar{\Vc}$, $\bar{\Yc}$,
$\bar{\Zc}$, $\bar{\Lambda}$ and $\bar{M}$ respectively for all
positive times. In what follows, we make the following assumption
regarding the feasibility of~\eqref{eq:sgf-i}.

\begin{assumption}\longthmtitle{Feasibility of
    $S_{\alpha}$}\label{as:li-separable}
  For all $i\in\{1,\dots,N\}$,~\eqref{eq:sgf-i} is feasible for all
  $(x_{i},y_{\Nc_i},z_{\Nc_i})\in\Pi_{x_i}\Fc\times
  \Pi_{y_{\Nc_i}}\bar{\Yc} \times \Pi_{z_{\Nc_i}}\bar{\Zc}$.
\end{assumption}

The following result gives a sufficient condition
under which Assumption~\ref{as:li-separable} holds.

\begin{lemma}\longthmtitle{Sufficient condition for feasibility of
    $S_{\alpha}$}\label{lem:suff-cond-feas}
  Suppose that the vectors
  $\{ \nabla g_i^k(x_i) \}_{k=1}^{p} \cup \{ \nabla h_i^l(x_i)
  \}_{l=1}^{q}$ are linearly independent for all
  $x_{i}\in\Pi_{x_i}\Fc$.  Then, for all
  $i\in\{1,\dots,N\}$,~\eqref{eq:sgf-i} is feasible for all
  $(x_{i},y_{\Nc_i},z_{\Nc_i})\in\Pi_{x_i}\Fc\times
  \Pi_{y_{\Nc_i}}\bar{\Yc} \times \Pi_{z_{\Nc_i}}\bar{\Zc}$.
\end{lemma}
\begin{proof}
  By considering the inequality constraints in~\eqref{eq:sgf-i} as
  equality constraints, and since $p+q\leq n$
  necessarily,~\eqref{eq:sgf-i} consists of a linear system of
  equations with at least as many unknowns as equations. If the number
  of equations is strictly less than the number of unknowns (i.e.,
  $p+q<n$),~\eqref{eq:sgf-i} is feasible. If the number of equations
  is equal to the number of unknowns, (i.e.,
  $p+q=n$),~\eqref{eq:sgf-i} is feasible because
  $\{ \nabla g_i^k(x_i) \}_{k=1}^{p} \cup \{ \nabla h_i^l(x_i)
  \}_{l=1}^{q}$ are linearly independent for all
  $x_{i}\in\Pi_{x_i}\Fc$.
\end{proof}

The next result establishes some feasibility and regularity properties
of $S_{\alpha}$. Its proof follows an argument analogous to the proof
of~\cite[Proposition 5.3]{AA-JC:24-tac}.

\begin{proposition}\longthmtitle{Well-posedness  and regularity of
    SP-SGF}\label{prop:properties-local-sgf}
  Under Assumption~\ref{as:li-separable}, the following statements
  hold:
  \begin{itemize}
  \item There exists an open neighborhood $U$ containing
    $\bar{\Vc}\times\bar{\Yc}\times\bar{\Zc}\times\bar{\Lambda}
    \times\bar{M}\times\Fc$ such that~\eqref{eq:primal-dual-sgf} is
    well-defined on $U$;
  \item The dynamical system~\eqref{eq:primal-dual-sgf} is locally
    Lipschitz on $U$;
  \item The Lagrange multipliers of~\eqref{eq:sgf-i} are unique and
    locally Lipschitz as a function of $x,y$ and $z$ on $U$.
  \end{itemize}
\end{proposition}

\section{Invariance and Convergence Analysis}

Having established the distributed character of the
algorithm~\eqref{eq:primal-dual-sgf}, here we show the forward
invariance of the feasible set and the asymptotic convergence to the
optimizer.

We start by introducing some useful notation. For $i\in\{1,\dots,N\}$,
$k\in\{1,\dots,p\}$ and $l\in\{1,\dots,q\}$, we let
$\psi_{y_i^k}(t;p_0),\psi_{z_j^l}(t;p_0),\psi_{x_i}(t;p_0)$ be the
solution
of~\eqref{eq:p-d-sgf-y},~\eqref{eq:p-d-sgf-z},~\eqref{eq:p-d-sgf-x}
respectively for initial conditions
$p_0=(v_0,y_0,z_0,\lambda_0,\mu_0,x_0)\in
\Pc:=\Vc\times\Yc\times\Zc\times\Lambda\times M \times\Fc$. We also
let
$\psi_{y}(t;p_0) = [\psi_{y_1^1}(t;p_0),
\dots,\psi_{y_1^p}(t;p_0),\dots,\psi_{y_N^1}(t;p_0),\dots,\psi_{y_N^p}(t;p_0)]$,
and define $\psi_{z}(t;p_0)$ and $\psi_{x}(t;p_0)$ analogously.  The
next result establishes the \textit{anytime} nature
of SP-SGF.

\begin{lemma}\longthmtitle{Anytime property}\label{lem:anytime}
  Suppose that $x_0\in\Fc$ and Assumption~\ref{as:li-separable} holds.
  Then, the trajectories of~\eqref{eq:primal-dual-sgf} satisfy
  $\psi_x(t;p_0)\in\Fc$ for all $t\geq0$.
\end{lemma}
\begin{proof}
  Since Assumption~\ref{as:li-separable} holds, the
  dynamics~\eqref{eq:primal-dual-sgf} are well-defined on a
  neighborhood $U$, cf.  Proposition~\ref{prop:properties-local-sgf}.
  If, at some $\bar{t}$,
  $\sum_{i=1}^{N}g_{i}^k(\psi_{x_i}(\bar{t};p_0))=0$ for
  $k\in\{1,\dots,p\}$, then because of the constraints
  in~\eqref{eq:sgf-i},
  \begin{align*}
    &\frac{d}{dt}\sum_{i=1}^N
      g_i^k(\psi_{x_i}(t;p_0))\rvert_{t=\bar{t}}
    \\
    &=\sum_{i=1}^N \nabla g_i^k(\psi_{x_i}(\bar{t};p_0))
      S_{\alpha}^i(\psi_{x_i}(\bar{t};p_0),\psi_{y_{\Nc_i}}(\bar{t};p_0),
      \psi_{z_{\Nc_i}}(\bar{t};p_0))  
    \\
    &\leq \! -\alpha\sum_{i=1}^N \Big( g_i^k(\psi_{x_i}(\bar{t};p_0)) +
      \! \sum_{j\in\Nc_i}
      (\psi_{y_i^k}(\bar{t};p_0)-\psi_{y_j^k}(\bar{t};p_0)) \Big) = 0 .
  \end{align*}
  Hence by Brezis' Theorem~\cite{HB:70}, it follows that
  $\sum_{i=1}^{N}g_{i}^k(\psi_{x_i}(t;p_0))\leq0$ for all $t\geq0$,
  $p_0\in\Pc$ and $k\in\{1,\dots,p\}$. By a similar argument, $\frac{d}{dt}\sum_{i=1}^N h_i^l(\psi_{x_i}(t;p_0))\rvert_{t=\bar{t}}=0$.
  Hence, it follows that $\sum_{i=1}^{N}h_{i}^l(\psi_{x_i}(t;p_0))=0$
  for all $t\geq0$, $p_{0}\in\Pc$ and $l\in\{1,\dots,q\}$.
\end{proof}

Next, we turn to the study of the convergence properties
of~\eqref{eq:primal-dual-sgf}. The next result establishes a connection
between the equilibrium points of $S_{\alpha}$ and the optimizers
of~\eqref{eq:distr-opt-pb-y-eps}.

\begin{proposition}\longthmtitle{Relationship between equilibria and
    optimizers}\label{prop:eq-sgf-optimizers}
  Let $x\in\Fc$. Then, $S_{\alpha}(x,y^{*,\epsilon},z^{*,\epsilon})=0$
  if and only if $x=x^{*,\epsilon}$.
\end{proposition}
\begin{proof}
  Note that
  $S_{\alpha}^{i}(x_i,y_{\Nc_i}^{*,\epsilon},z_{\Nc_i}^{*,\epsilon})$
  is the safe gradient flow associated to the optimization
  problem~\eqref{eq:local-distr-opt-pb-y-eps}, which by
  Proposition~\ref{prop:equivalence-constraint} has
  $x_{i}^{*,\epsilon}$ as the unique optimizer.  The result then
  follows from Proposition~\ref{properties-sgf}(v).
\end{proof}

\smallskip

The next result states that the trajectories of~$x$
in~\eqref{eq:primal-dual-sgf} converge to the optimizer
of~\eqref{eq:distr-opt-pb-y-eps}.

\begin{theorem}\longthmtitle{Convergence to
    optimizer}\label{prop:convergence-to-opt}
  Suppose Assumption~\ref{as:li-separable} holds.  For any $\delta>0$
  and compact set $\Omega$ containing
  $\setdef{x\in\Fc}{\norm{x-x^{*,\epsilon}}\leq\delta}$, there exists
  $\tau_{\delta,\Omega}>0$ and $T_{\delta,\Omega}$ such that if
  $\tau<\tau_{\delta,\Omega}$, then under the
  dynamics~\eqref{eq:primal-dual-sgf}, $\norm{\psi_x(t;p_0)-x^{*,\epsilon}}<\delta$
  for all $t\geq T_{\delta,\Omega}$ and
  $p_0=(v_0,y_0,z_0,\lambda_0,\mu_0,x_0)\in\Pc$.  Moreover, if $\Fc$
  is bounded, then for any $\tau>0$, $\lim\limits_{t\to\infty}\norm{\psi_x(t;p_0)-x^{*,\epsilon}} = 0$
  for all $p_{0}\in\Pc$.
\end{theorem}
%
%
\begin{proof}
  Since the dynamics in~\eqref{eq:primal-dual-sgf} are not
  differentiable, the standard version of Tikhonov's theorem for
  singular perturbations~\cite[Theorem~11.2]{HKK:02} is not
  applicable.  Instead we use~\cite[Corollary~3.4]{FW:05}, which gives
  a Tikhonov-type singular perturbation statement for differential
  inclusions. In the case of non-smooth ODEs for which the fast
  dynamics do not depend on the slow variable,
  like~\eqref{eq:primal-dual-sgf}, we need to check the following
  assumptions.  First, that the dynamics~\eqref{eq:primal-dual-sgf}
  are Lipschitz. Note that local Lipschitzness
  of~\eqref{eq:primal-dual-sgf} follows from
  Proposition~\ref{prop:properties-local-sgf}, the Lipschitzness of
  the gradients of $f$ and $g$ and the Lipschitzness of the $\max$
  operator. Moreover, since $\bar{\Vc}$, $\bar{\Yc}$, $\bar{\Zc}$,
  $\bar{\Lambda}$, $\bar{M}$ and $\Omega$ are compact, we can redefine
  the dynamics~\eqref{eq:primal-dual-sgf} outside of
  $\bar{\Vc}\times\bar{\Yc}\times\bar{\Zc}\times\bar{\Lambda}\times\bar{M}\times\Omega$
  so that they are globally Lipschitz while still keeping the same
  dynamics for initial conditions in
  $\Vc\times\Yc\times\Zc\times\Lambda\times M\times\Omega$.
  Second, existence and uniqueness of the equilibrium of the fast
  dynamics. This follows from the fact
  that~\eqref{eq:distr-opt-pb-y-eps} has a strongly convex objective
  function and convex constraints, which implies that it has a unique
  KKT point. 
  Third, Lipschitzness and asymptotic stability of the reduced-order
  model
  \begin{align}\label{eq:rom}
    \dot{\bar{x}}=S_{\alpha}(\bar{x},y^{*,\epsilon},z^{*,\epsilon})
  \end{align}
  Lipschitzness follows from
  Proposition~\ref{prop:properties-local-sgf} and asymptotic stability
  follows from~\cite[Theorem~5.6]{AA-JC:24-tac} and the fact that
  $\Omega$ is compact. Fourth, asymptotic stability of the fast
  dynamics. This follows
  from~\cite[Theorem~5.1]{AC-EM-SHL-JC:18-tac}. Finally, note that
  $x^{*,\epsilon}$ is the only equilibrium point of~\eqref{eq:rom} and
  hence the result follows from~\cite[Corollary~3.4]{FW:05}.

  Now suppose that $\Fc$ is compact. Pick an arbitrary
  $\theta>0$. Since $f$ is continuous and $x^{*,\epsilon}$ is the
  unique minimizer of~\eqref{eq:distr-opt-pb-y-eps}, there exist
  constants $a_{\theta}>0$, $b_{\theta}>0$ such that the sets
  \begin{align*}
    A_{\theta}&=\setdef{x\in\Fc}{f(x)-f(x^{*,\epsilon})\leq a_{\theta}} \\
    B_{\theta}&=\setdef{x\in\Fc}{\norm{x-x^{*,\epsilon}}\leq b_{\theta}} \\
    C_{\theta}&=\setdef{x\in\Fc}{\norm{x-x^{*,\epsilon}}\leq\theta}
  \end{align*}
  satisfy $ B_{\theta} \subseteq A_{\theta} \subseteq C_{\theta}$.
  Next, we show that there exists $T_{\theta}>0$ such that
  $\psi_x(t;p_0)\in C_{\theta}$ for $t\geq T_{\theta}$ and all
  $p_{0}\in\Pc$ (i.e., $C_{\theta}$ is asymptotically stable relative
  to $\Fc$). Since $\theta$ is arbitrary, this completes the proof.
  Let
  $( \{ \phi_i^k(x_i,y_{\Nc_i},z_{\Nc_i}) \}_{k=1}^{p}, \{
  \chi_i^l(x_i,y_{\Nc_i},z_{\Nc_i}) \}_{l=1}^q )$ be the Lagrange
  multipliers associated to the optimization problem defining
  $S_{\alpha}^{i}(x_i,y_{\Nc_i},z_{\Nc_i})$, which are unique and
  locally Lipschitz by
  Proposition~\ref{prop:properties-local-sgf}. Then, by following an
  argument analogous to the one in the proof
  of~\cite[Lemma~5.8]{AA-JC:24-tac}:
  \begin{align}\label{eq:df-dt}
    & \frac{d}{dt}(f(x)-f(x^{*,\epsilon})) \leq
      -\norm{S_{\alpha}(x,y,z)}^2
    \\
    \notag
    &\quad +\sum_{i=1}^N \sum_{k=1}^p
      \phi_i^k(x_i,y_{\Nc_i},z_{\Nc_i}) \alpha
      (g_i^k(x_i)+\sum_{j\in\Nc_i}(y_i^k-y_j^k))
    \\
    \notag
    &\quad +\sum_{i=1}^N \sum_{l=1}^q
      \chi_i^l(x_i,y_{\Nc_i},z_{\Nc_i}) \alpha
      (h_i^l(x_i)+\sum_{j\in\Nc_i}(z_i^k-z_j^k)). 
  \end{align}
  Since~\eqref{eq:p-d-sgf-v}-\eqref{eq:p-d-sgf-mu} are the projected
  saddle-point dynamics of~\eqref{eq:distr-opt-pb-y-eps} and the
  objective function of~\eqref{eq:distr-opt-pb-y-eps} is strongly
  convex, by~\cite[Theorem 5.1]{AC-EM-SHL-JC:18-tac}, the variables
  $v, y, z, \lambda, \mu$ converge to the KKT point
  of~\eqref{eq:distr-opt-pb-y-eps} for all $\tau>0$. Moreover, since
  $S_{\alpha}(x,y^{*,\epsilon},z^{*,\epsilon})=0$ if and only if
  $x=x^{*,\epsilon}$ by Proposition~\ref{prop:eq-sgf-optimizers},
  $S_{\alpha}$ is continuous by
  Proposition~\ref{prop:properties-local-sgf} and $\Pc$ is compact,
  for any fixed $\tau>0$, there exist $\sigma_{\theta,\tau}$ and
  $T_{1,\theta,\tau}$ such that for all $t\geq T_{1,\theta,\tau}$ and
  $p_{0}\in\Pc$, $\norm{\psi_x(t;p_0)-x^{*,\epsilon}} > b_{\theta}$
  implies
  $\norm{S_{\alpha}(\psi_x(t;p_0),\psi_y(t;p_0),\psi_z(t;p_0))}>\sigma_{\theta,\tau}$.

  Now, define
  \begin{align*}
    \hat{g}_i^k(t,p_0)
    &=
      g_i^k(\psi_{x_i}(t;p_0)) +
      \sum_{j\in\Nc_i}(\psi_{y_i^k}(t;p_0)-\psi_{y_j^k}(t;p_0)),   
    \\
    \hat{h}_i^l(t,p_0)
    &=
      h_i^l(\psi_{x_i}(t;p_0))+\sum_{j\in\Nc_i}(\psi_{z_i^l}(t;p_0)-\psi_{z_j^l}(t;p_0)), 
    \\
    \hat{\phi}_i^k(t,p_0)
    &=
      \phi_i^k(\psi_{x_i}(t;p_0),\psi_{y_{\Nc_i}}(t;p_0),\psi_{z_{\Nc_i}}(t;p_0)), 
    \\
    \hat{\chi}_i^l(t,p_0)
    &=
      \chi_i^l(\psi_{x_i}(t;p_0),\psi_{y_{\Nc_i}}(t;p_0),\psi_{z_{\Nc_i}}(t;p_0)), 
  \end{align*}
  and let us show that there exists a time $T_{2,\theta,\tau}>0$ such
  that
  \begin{align}\label{eq:sigma-theta}
    \alpha\sum_{i=1}^N \sum_{k=1}^p
    \hat{\phi}_i^k(t,p_0)\hat{g}_i^k(t,p_0)+\alpha\sum_{i=1}^N \sum_{l=1}^q
    \hat{\chi}_i^l(t,p_0) \hat{h}_i^l(t,p_0)<\frac{\sigma_{\theta,\tau}}{2},
  \end{align}
  for all $t\geq T_{2,\theta}$ and $p_{0}\in\Pc$.
  First define
  \begin{align*}
    c_{\phi} :=
    \max_{\substack{(x,y,z)\in\Fc\times\bar{\Yc}\times\bar{\Zc}\\i\in\{1,\dots,N\}\\k\in\{1,\dots,p\}}} 
    |\phi_i^k(x_i,y_{\Nc_i},z_{\Nc_i})|, 
    \\
    c_{\chi} :=
    \max_{\substack{(x,y,z)\in\Fc\times\bar{\Yc}\times\bar{\Zc}\\i\in\{1,\dots,N\}\\l\in\{1,\dots,q\}}} 
    |\chi_i^l(x_i,y_{\Nc_i},z_{\Nc_i})|.
  \end{align*}
  Note that such $c_{\phi}, c_{\chi}$ exist because $\Fc$, $\bar{\Yc}$
  and $\bar{\Zc}$ are compact.  Now note that
  \begin{align*}
    &\frac{d}{dt} (\hat{g}_i^k(t,p_0)) \leq -\alpha \hat{g}_i^k(t,p_0)+\sum_{j\in\Nc_i}(\dot{\psi}_{y_i^k}(t;p_0)-\dot{\psi}_{y_j^k}(t)),
    \\
    &\frac{d}{dt} (\hat{h}_i^l(t,p_0)) \leq -\alpha \hat{h}_i^l(t,p_0)+\sum_{j\in\Nc_i}(\dot{\psi}_{z_i^l}(t;p_0)-\dot{\psi}_{z_j^l}(t)). 
  \end{align*}
  Since the variables $y_{i}^{k}, z_{i}^{k}$ are convergent
  by~\cite[Theorem 5.1]{AC-EM-SHL-JC:18-tac},
  \begin{align*}
    \lim_{t\to\infty} \dot{\psi}_{y_i^k}(t;p_0) = 0, \quad \forall
    i\in\{1,\dots,N\}, \ k\in\{1,\dots,p\},
    \\
    \lim_{t\to\infty} \dot{\psi}_{z_i^l}(t;p_0) = 0, \quad \forall
    i\in\{1,\dots,N\}, \ l\in\{1,\dots,q\}. 
  \end{align*}
  for all $p_{0}\in\Pc$. Hence, there exists a time
  $\hat{T}_{2,\theta,\tau}>0$ such that
  \begin{align*}
    \sum_{j\in\Nc_i}(\dot{\psi}_{y_i^k}(t;p_0)-\dot{\psi}_{y_j^k}(t))
    \leq \frac{\sigma_{\theta,\tau}}{8 \alpha Np c_{\phi}},
    \\
    \sum_{j\in\Nc_i}(\dot{\psi}_{z_i^l}(t;p_0)-\dot{\psi}_{z_j^l}(t))
    \leq \frac{\sigma_{\theta,\tau}}{8 \alpha Np c_{\chi}} 
  \end{align*}
  for all $i\in\{1,\dots,N\}$, $k\in\{1,\dots,p\}$,
  $l\in\{1,\dots,q\}$ and $t\geq\hat{T}_{2,\theta,\tau}$,
  $p_{0}\in\Pc$. By the Comparison Lemma~\cite[Lemma 3.4]{HKK:02}, it
  holds that
  \begin{align*}
    &\hat{g}_i^k(t,p_0) \leq \hat{g}_i^k(\hat{T}_{2,\theta,\tau},p_0)
      e^{-\alpha (t-\hat{T}_{2,\theta,\tau})} +
      \frac{\sigma_{\theta,\tau}}{8 Np c_{\phi}},
    \\
    &\hat{h}_i^l(t,p_0) \leq \hat{h}_i^l(\hat{T}_{2,\theta,\tau},p_0)
      e^{-\alpha (t-\hat{T}_{2,\theta,\tau})} +
      \frac{\sigma_{\theta,\tau}}{8 Np c_{\chi}}. 
  \end{align*}
  Since $\Fc$ is compact, $\psi_{x}(t;p_0)\in\Fc$ by
  Lemma~\ref{lem:anytime}, $\psi_{y}(t;p_0)\in\bar{\Yc}$ and
  $\psi_{z}(t;p_0)\in\bar{\Zc}$ for all $t\geq0$, this implies that
  there exists a time $T_{2,\theta,\tau}>0$ such
  that~\eqref{eq:sigma-theta} holds for all $t\geq T_{2,\theta,\tau}$
  and $p_{0}\in\Pc$. Now, let
  $T_{\theta,\tau}=\max\{T_{1,\theta,\tau}, T_{2,\theta,\tau}
  \}$. Then, it holds that for all $t\geq T_{\theta,\tau}$, $\frac{d}{dt}(f(\psi_x(t;p_0))-f(x^{*,\epsilon})) < 0$ if $\norm{\psi_x(t;p_0)-x^{*,\epsilon}}>b_{\theta}$. Since
  $B_{\theta}\subseteq
  A_{\theta}$, this implies that
  $A_{\theta}$ is asymptotically stable relative to $\Fc$. Since
  $A_{\theta}\subseteq C_{\theta}$, it follows that $C_{\theta}$ is
  asymptotically stable relative to $\Fc$, hence completing the
  proof. Note that this argument is valid for all fixed $\tau>0$.
\end{proof}

\smallskip

By Theorem~\ref{prop:convergence-to-opt}, the trajectories of the $x$
variable in SP-SGF converge arbitrarily close to the optimizer
$x^{*,\epsilon}$ provided that the timescale parameter $\tau$ is small
enough. Moreover, if the feasible set $\Fc$ is bounded, asymptotic
convergence holds for any timescale. The combination of the scalable
and distributed character, cf. Remark~\ref{rem:scalable-distr}, the
anytime nature, cf. Lemma~\ref{lem:anytime}, and the convergence
properties, cf. Theorem~\ref{prop:convergence-to-opt} means that
SP-SGF provides the algorithmic solution with the properties stated in
Section~\ref{sec:problem-statement}.

\begin{example}\longthmtitle{Resource allocation}
  We illustrate the behavior of SP-SGF in a resource allocation
  example.  Consider $13$ agents whose communication graph is an
  undirected line graph. Solving distributed optimization problems
  with this particular topology is challenging due to its low
  connectivity.  Each agent's state variable is
  $x_{i}=[x_{i,1}, x_{i,2}]\in\real^{2}$, where $x_{i,1}$
  (resp. $x_{i,2}$) corresponds to the amount of resource $1$
  (resp. $2$) allocated by agent $i$. Resource $1$ is subject to an
  equality constraint and resource $2$ is subject to an inequality
  constraint. Hence, the agents solve the optimization problem,
  \begin{align}\label{eq:separable-example}
    & \min \limits_{{\{x_i\}_{i=1}^{13}}} \sum_{i=1}^{13} \frac{1}{2}\norm{x_i}^2,
    \\
    \notag
    & \text{s.t.} \ h(\{x_i\}_{i=1}^{13}) = 5-\sum_{i=1}^{13}p_i x_{i,1} = 0, \\
    \notag
    & \quad \ g(\{x_i\}_{i=1}^{13}) = -3+\sum_{i=1}^{13} e^{-x_{i,2}} \leq 0.
  \end{align}
  with $p_1=1$, $p_{2}=3$, $p_{3}=2$, $p_{4}=1$, $p_{5}=1$,
  $p_{6}=1$, $p_{7}=2$, $p_{8}=4$, $p_{9}=1$, $p_{10}=1$, $p_{11}=0.5$, $p_{12}=2$, $p_{13}=1$.  Note that the condition
  in Lemma~\ref{lem:suff-cond-feas} holds and hence
  Assumption~\ref{as:li-separable} holds. This implies by
  Proposition~\ref{prop:convergence-to-opt}
  that~\eqref{eq:primal-dual-sgf} is well-defined
  for~\eqref{eq:separable-example}. We use $\epsilon=0.0001$ and
  $\alpha=1$.  Figure~\ref{fig:sim-variables-separable} illustrates
  the convergence of the $x$ variables under SP-SGF.

  Since the feasible set of~\eqref{eq:separable-example} is unbounded,
  Proposition~\ref{prop:convergence-to-opt} states that convergence
  arbitrarily close to the optimizer can be achieved by taking $\tau$
  sufficiently small.  Figure~\ref{fig:x1-taus} illustrates the
  convergence of the quantities $\sum_{i=1}^{13}x_{i,1}^{2}$ and $\sum_{i=1}^{13}x_{i,2}^{2}$ for
  different values of $\tau$ and shows that this quantity converges
  exactly to its optimal value for a wide range of values of $\tau$,
  suggesting that the statement in
  Theorem~\ref{prop:convergence-to-opt} might be too conservative.
  Figure~\ref{fig:sim-constraint-satisfaction-separable} compares the
  evolution of the constraints of~\eqref{eq:separable-example} under
  SP-SGF against two other algorithms: the projected saddle point
  dynamics~\eqref{eq:projected-saddle-point-dyn} (abbreviated SP),
  which is not distributed, and the projected saddle-point dynamics
  (abbreviated SP-CM) for its reformulation with constraint mismatch
  variables as in~\eqref{eq:distr-opt-pb-y}, which is distributed.
  SP-SGF satisfies the constraints at all times whereas SP and SP-CM
  do not. We note that, in this case, SP-SGF requires running a
  dynamical system with 104 scalar variables (8 for each agent), SP-CM
  requires running a dynamical system of 78 scalar variables (6 for
  each agent) and SP requires running a dynamical system with 28
  scalar variables. \oprocend
\end{example}

\begin{figure}[htb]
  \centering
  \includegraphics[width=0.3\textwidth]{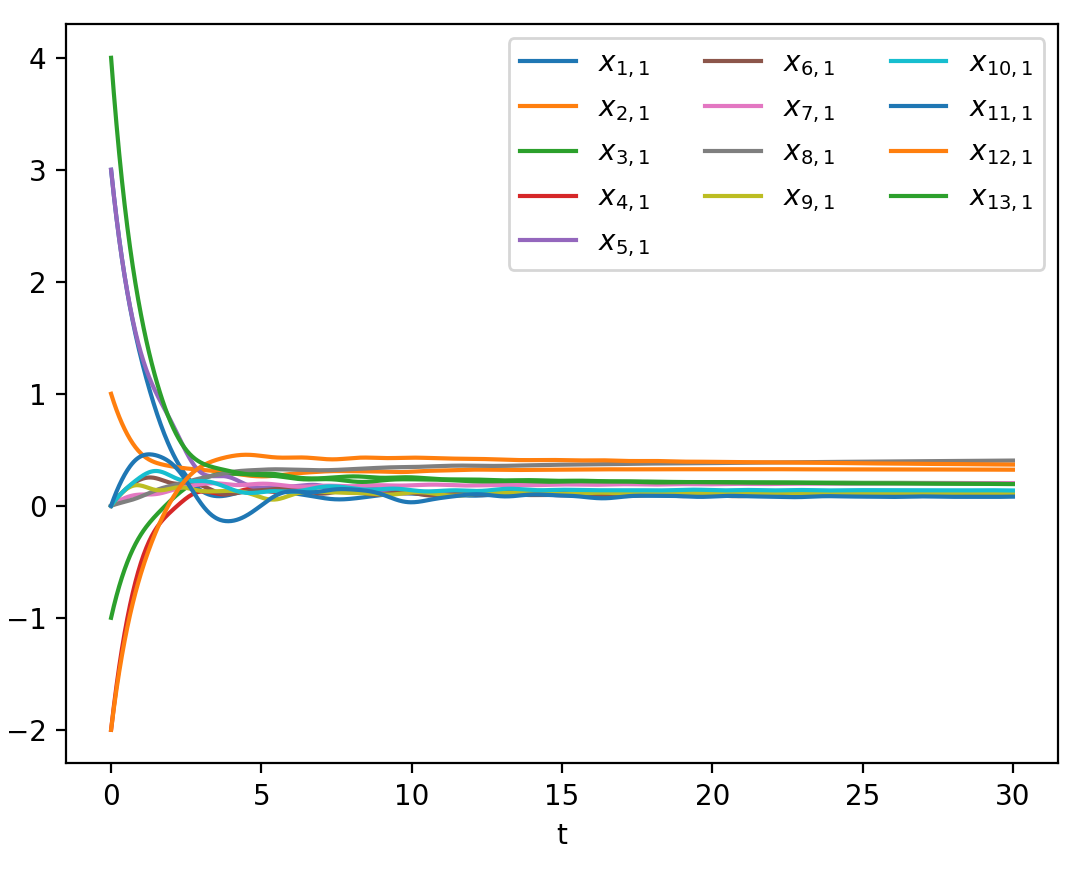}
  \\
  \includegraphics[width=0.3\textwidth]{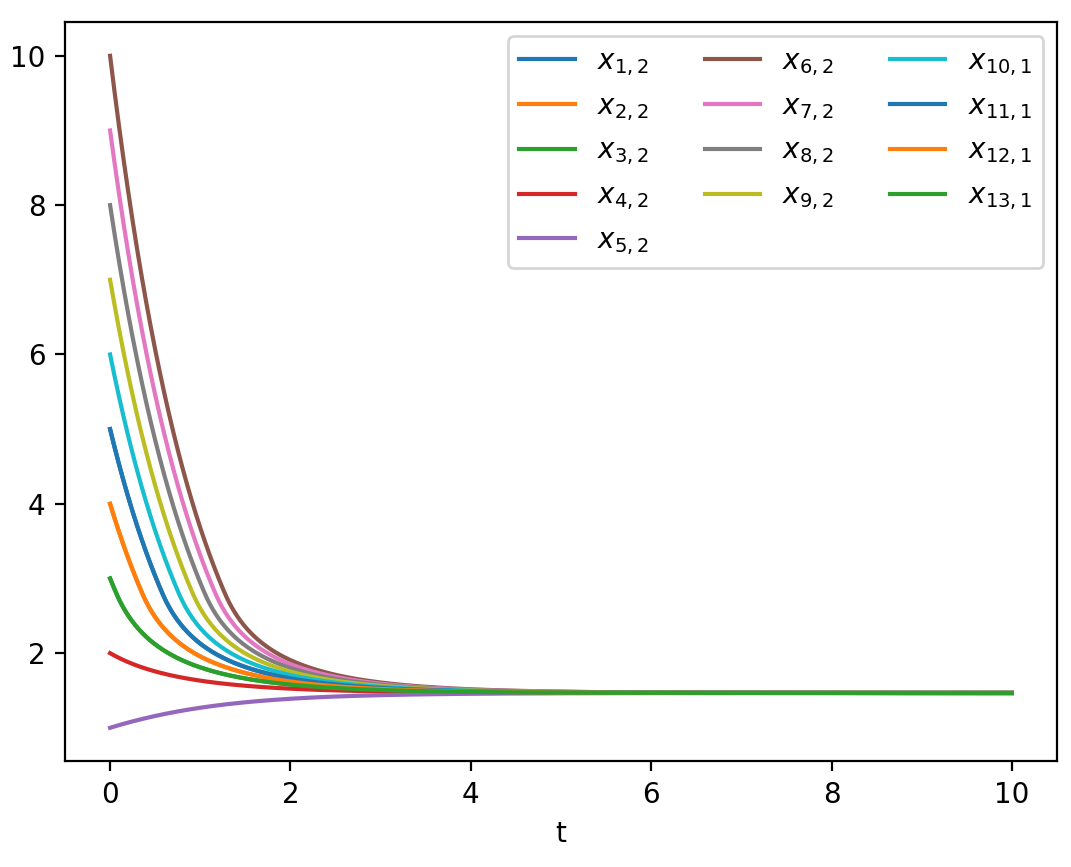}
  \caption{Evolution of $x_{1,i}$ (top) and $x_{2,i}$
    (bottom) for $i\in\until{13}$ under SP-SGF
    for~\eqref{eq:separable-example} with initial conditions
    $x_{1}=[3, 5]$, $x_{2}=[1, 4]$, $x_{3}=[-1, 3]$, $x_{4}=[-2, 2]$,
    $x_{5}=[3, 1]$, $x_{6}=[0, 10]$, $x_7=[0, 9]$, $x_8=[0, 8]$,
    $x_9=[0,7]$, $x_{10}=[0,6]$, $x_{11}=[0,5]$, $x_{12}=[-2,4]$, $x_{13}=[4,3]$, $v_{i,1}=v_{i,2}=z_{i}=y_{i}=\lambda_{i}=\mu_{i}=0$ for
    all $i\in\until{13}$ and $\tau=1$.}\label{fig:sim-variables-separable}
\end{figure}
%
%

\begin{figure}[htb]
  \centering
  \includegraphics[width=0.3\textwidth]{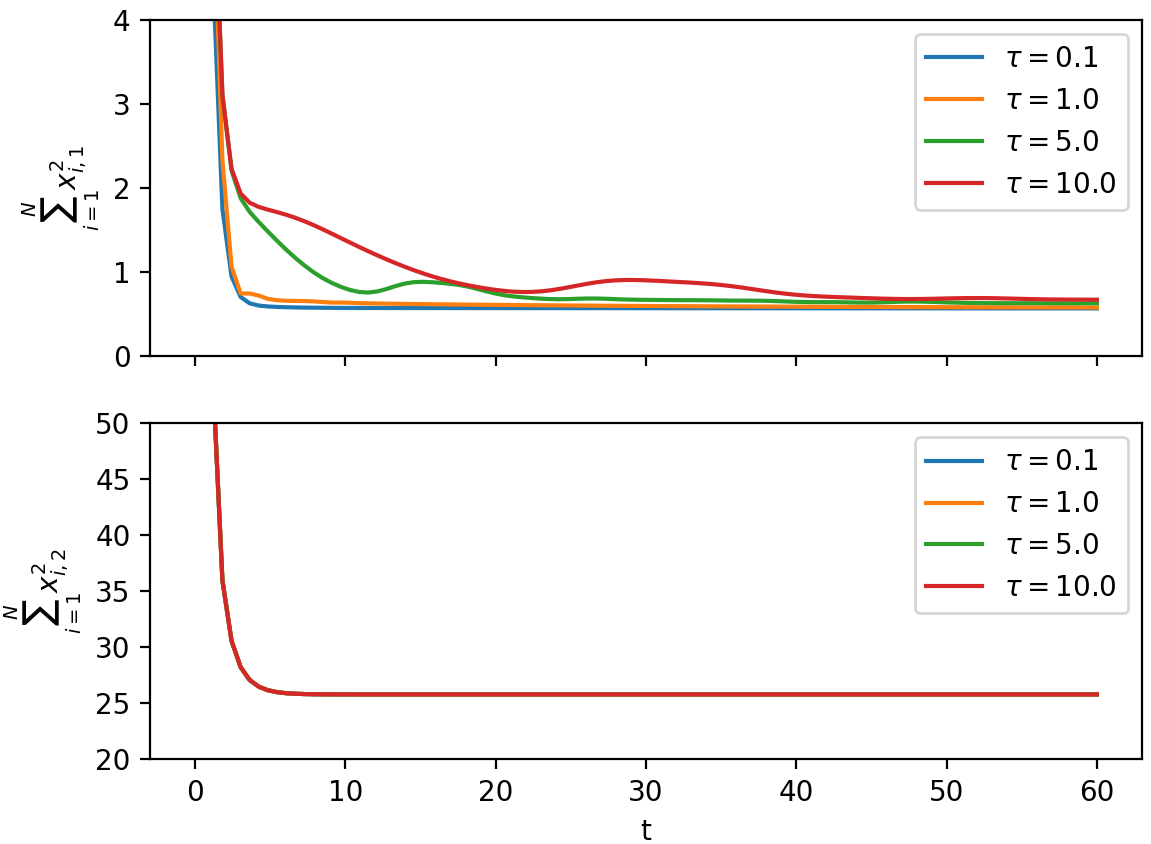}
  \caption{This plot shows the evolution of $\sum_{i=1}^{13}x_{i,1}^{2}$ and $\sum_{i=1}^{13}x_{i,2}^{2}$ under SP-SGF
    with initial conditions as in Figure~\ref{fig:sim-variables-separable}
    for different values of
    $\tau$.}\label{fig:x1-taus}
\end{figure}
%
%

\begin{figure}[htb]
  \centering
  \includegraphics[width=0.3\textwidth]{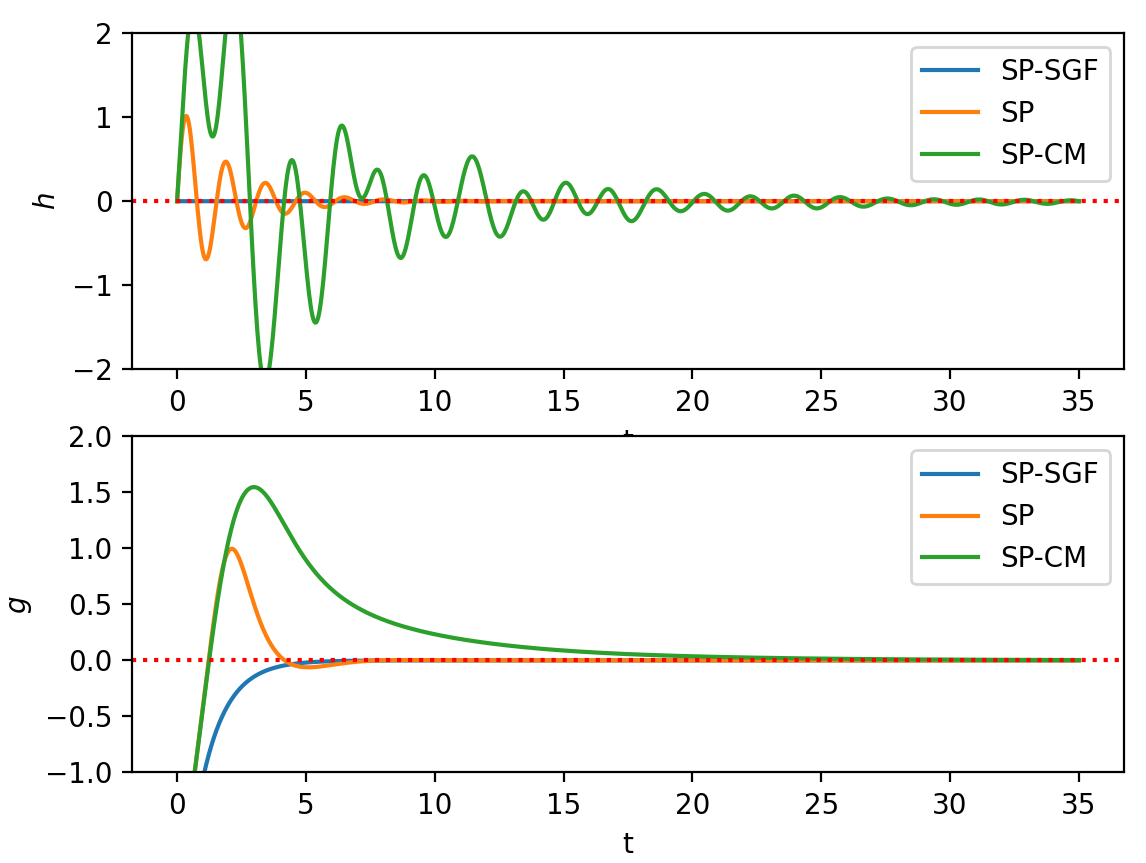}
  \caption{This plot shows the evolution of the constraints
    of~\eqref{eq:separable-example} for SP-SGF with the same initial
    conditions as in Figure~\ref{fig:sim-variables-separable}, SP with
    the same primal initial conditions as in
    Figure~\ref{fig:sim-variables-separable} and $\lambda=\mu=0$ and
    SP-CM with the same initial conditions as in
    Figure~\ref{fig:sim-variables-separable} for $x_{i}$, $z_{i}$,
    $y_{i}$, $\lambda_{i}$ and $\mu_{i}$ for
    $i\in\until{13}$.}\label{fig:sim-constraint-satisfaction-separable}
\end{figure}

\section{Conclusions}\label{sec:conclusions}
We have introduced a continuous-time dynamical system that solves
network optimization problems with separable objective function and
constraints in a distributed and anytime fashion. We have achieved
this by combining the projected saddle-point dynamics and the safe
gradient flow in a cascaded system. We have argued the scalable nature
of the algorithm execution from the point of view of individual agents
and established practical convergence to the optimizer when the
feasible is unbounded, and exact convergence when it is bounded.
Future work will consider other network optimization problems, refine
the convergence guarantees presented here and possibly design new
distributed, anytime algorithms, and investigate discretization
schemes for the continuous-time dynamics. We also plan to apply our
coordination algorithms in the implementation of optimization-based controllers arising from safety certificates for
multi-agent systems.

\bibliography{../../bib/alias,../../bib/JC,../../bib/Main-add,../../bib/Main} 
\bibliographystyle{IEEEtran}

%
%
%
%

\end{document}